\documentclass[reqno]{amsart} 
\usepackage[utf8]{inputenc} 
\usepackage{amsthm,amsmath,amssymb,amsfonts} 
\usepackage{tikz} 
\usepackage{graphicx} 
\usepackage{psfrag} 
\usepackage{cancel} 
\usepackage{hyperref} 
\usepackage{comment} 
\usepackage{enumerate} 
\newtheorem{theorem}{Theorem}[section] 
\newtheorem{lemma}{Lemma}[section] 
\newtheorem{proposition}{Proposition}[section] 
\newtheorem*{corollary*}{Corollary} 
\newtheorem{corollary}{Corollary}[section] 
\newtheorem{remark}{Remark}[section] 
 
\numberwithin{equation}{section}

\newcommand{\card}{\ensuremath{\#}}

\theoremstyle{plain}
\newtheorem{maintheorem}{Theorem}

\title[]{Maps in Dimension One with Infinite Entropy}
\thanks{
This work has been partially supported by 
``Projeto Tem\'atico Din\^amica em Baixas Dimens\~oes'' FAPESP Grant 2011/16265-2, FAPESP Grant 2015/17909-7
}

\author{Peter Hazard} 
\address{Peter Hazard, Instituto de Matem\'{a}tica e Estat{\'i}stica, USP, S\~{a}o Paulo, SP, Brazil}
\email[]{pete@ime.usp.br}

\date{\today}

\begin{document}

%\paragraph{Comments.}
%The following are some changes that have been made or need making.
%\begin{enumerate}
%\item 
%\end{enumerate}

%\newpage

\begin{abstract}
We give examples of endomorphisms in dimension one with infinite topological entropy which are $\alpha$-H\"older and $(1,p)$-Sobolev for all $0\leq \alpha< 1$ and 
$1\leq p<\infty$. 
This is constructed within a family of endomorphisms with infinite topological entropy and which traverse all $\alpha$-H\"older and $(1,p)$-Sobolev classes.
Finally, we also give examples of endomorphisms, also in dimension one, which lie in the big and little Zygmund classes, answering a question of M.\ Benedicks.
\end{abstract}

\maketitle

\section{Introduction}
\subsection{Background.} 
%%
%In 1965, 
Adler, Konheim and McAndrew~\cite{AKM} defined the 
{\it topological entropy} of a continuous self-mapping of a compact metric space as an analogue of the 
Kolmogorov-Sinai {\it metric entropy} of a measure-preserving transformation of a measure space.
Recall that the topological entropy of a continuous self-mapping is a non-negative real number, possibly infinite, which is invariant under topological conjugacy. 
In~\cite{AKM}, an example of a map for which the topological entropy is infinite was given (a full shift on infinitely many symbols).
(See also~\cite{WaltersBook,DGS}.)
In~\cite{Yano1980}, it was shown that even on smooth manifolds 
there exist examples of continuous self-mappings with infinite topological entropy.
In fact, a stronger statement was shown: for smooth compact manifolds of dimension two or greater, a generic homeomorphism 
(with respect to the uniform topology) has infinite topological entropy.
 
In contrast, 
self-mappings with sufficient regularity or smoothness must have finite 
topological entropy.
More precisely~\cite[Theorem 3.2.9]{KandH}, 
if $f$ is a Lipschitz self-mapping of the compact metric space 
$(X,d)$ with finite Hausdorff dimension $D(X)$,
we have the following inequality
\begin{equation}
h_\mathrm{top}(f)\leq D(X) \log^+ \mathrm{Lip}_d(f)
\end{equation}
where 
$h_\mathrm{top}(f)$ denotes the topological entropy of $f$, and
$\mathrm{Lip}_d(f)$ denotes the Lipschitz constant of $f$ with respect to the metric $d$.
(See also~\cite{Ito1970,Bowen1971}.)

In~\cite{dFHazT2}, an investigation was started into what occurs between $C^0$ and Lipschitz regularity, 
in the case of homeomorphisms on smooth compact manifolds.
The notion of `between' can be taken in several different directions.
For compact subsets of the real line, for example, 
given $0\leq \alpha<\beta<1$, 
if 
$C^\alpha$ denotes the space of $\alpha$-H\"older self-maps and 
$C^Z$ and $C^{\mathrm{Lip}}$ denote the spaces of self-maps satisfying respectively the Zygmund and Lipschitz conditions 
\footnote{See~\cite{ZygmundBook} or Sections~\ref{subsect:not+term} and~\ref{sect:Zygmund} for definitions and basic properties.}
then
\begin{equation}
C^0
\supset C^\alpha
\supset C^\beta
\supset C^{Z}
\supset C^{\mathrm{Lip}} \ .
\end{equation}
Similarly, if 
$\mathbb{DAE}$ denotes the space of continuous self-maps differentiable (Lebesgue)-almost everywhere,
$\mathbb{AC}$ denotes the space of continuous self-maps which are absolutely continuous,
$\mathbb{BV}$ denotes the space of continuous self-maps with bounded variation, and 
$\mathbb{W}^{1,p}$, $1\leq p\leq \infty$, denotes the space continuous self-maps satisfying the $W^{1,p}$-Sobolev condition, then
\begin{equation}
C^0=UC
\supset \mathbb{DAE}
\supset \mathbb{BV}
\supset \mathbb{AC} \simeq \mathbb{W}^{1,1}
\supset \mathbb{W}^{1,p}
\supset \mathbb{W}^{1,\infty}\simeq C^{\mathrm{Lip}} \ .
\end{equation}
With some care, (most of) these regularity classes, and the associated inclusions, can be extended to higher dimensions, and 
even to general smooth manifolds.
In~\cite{dFHazT2} an investigation into which values of entropy are possible in these two families of inclusions was initiated.
It was shown that for any $\alpha\in[0,1)$ and $p\in[1,\infty)$,
infinite entropy was not only possible, but a generic property in certain spaces (suitably topologised) of bi-$\alpha$-H\"older homeomorphisms, 
and of bi-$(1,p)$-Sobolev homeomorphisms, on smooth manifolds of dimension two or greater.
%%%%%%%%%%%%%%%
\begin{comment}
More precisely, for any compact manifold $M$ of dimension two or more, 
if 
$\mathcal{H}^1_\alpha(M)$ denotes the closure of the space of bi-Lipschitz homeomorphisms in the $\alpha$-H\"older-Whitney topology;
$\mathcal{S}^{1,p}(M)$ denotes the space of bi-$W^{1,p}$-Sobolev homeomorphisms endowed with the $(1,p)$-Sobolev-Whitney topology;
then in both of these space, for any $\alpha\in [0,1)$ and $p\in [1,\infty)$, infinite topological entropy is a generic property.
We note that the endomorphism case is still open.
\end{comment}
%%%%%%%%%%%%%%
%
%%%%%%%%%%%%%%%
\begin{comment}
Following~\cite{dFHazT2}, C.\ Tresser asked the author the following question:
%
\begin{quote}
{\it Tresser's Question:}
Is there some regularity class of mappings between $C^0$ and Lipschitz 
in which infinite topological entropy is not a generic property, but 
infinite entropy mappings exist?
\end{quote}
%
We note that this is currently unknown even in the H\"older case. 
(This was the original case of interest.)
\end{comment}
%%%%%%%%%%%%%
%
Note that the result in~\cite{dFHazT2} is only for the closure of bi-Lipschitz maps in the appropriate topology.
M.\ Benedicks asked the following question:
\begin{quote}
{\it Benedicks' Question:}
Is there a mapping in the big Zygmund class with infinite topological entropy?
In the little Zygmund class?
\end{quote}
Here we give an answer to these questions in the case of endomorphisms on compact one-manifolds.
Specifically, we will work on the closed unit interval, but the generalisation to the circle case will follow immediately.
(Note that homeomorphisms on compact one-manifolds must have zero entropy.)
%%%%%%%%%%%%%%%
\begin{comment}
It will become clear that these two questions are actually related: 
the answer to the second question will also give us an answer to the first.
\end{comment}
%%%%%%%%%%%%%
%%%%%%%%%%%%%%%%%%%%%%%%%%%%%%%%%%%%
\subsection{Summary of results.}\label{subsect:summary_of_results}
%%%%%%%%%%%%%%%%%%%%%%%%%%%%%%%%%%%%
First we construct examples of endomorphisms with infinite topological entropy lying in a H\"older or Sobolev class. 
\begin{maintheorem}\label{mainthm:HolderSobolev}
There exists a continuous one-parameter family of endomorphisms $f_a\in C([0,1],[0,1])$, $a\in (0,1]$, 
with the following properties
\begin{enumerate}
\item
for all $a\in (0,1]$ 
\begin{enumerate}
\item
all $f_a$ are topologically conjugate
\item
$h_\mathrm{top}(f_a)=+\infty$
\item
$f_a$ is not expansive, $h$-expansive or asymptotically $h$-expansive
\end{enumerate}
\item
for $a=1$
\begin{enumerate}
\item
$f_a$ has modulus of continuity $\omega(t)=t\log (1/t)$
\item
$f_a$ is in the Sobolev class $W^{1,p}$ for $1\leq p<\infty$
\item
Lebesgue measure is a measure of maximal entropy for $f_a$ (though there are at least countably many such measures)
\end{enumerate}
\item
for $a\in (0,1)$
\begin{enumerate}
\item
$f_a$ is $C^\alpha$ if and only if $\alpha\leq a$.
\item
$f_a$ is $W^{1,p}$ if and only if $p<(1-a)^{-1}$
\item
Lebesgue measure is not preserved by $f_a$, but there exist measures of maximal entropy for $f_a$ which are absolutely continuous with respect to Lebesgue.
\end{enumerate}
\end{enumerate}
\end{maintheorem}
We note that, for the specific examples considered here, 
half of the work is already done by
Morrey's inequality: 
namely, if $f_a$ lies $W^{1,p}$ then it automatically follows that $f_a$ is $C^\alpha$, where $\alpha=1-\frac{1}{p}$.
{\it However}, 
we also give an explicit proof of Theorem~\ref{mainthm:HolderSobolev}(3)(a). 
The reason being that our construction is made using piecewise-affine horseshoes as `model maps' from which the construction is made.
If the model map which we start with is H\"older but not, for instance, differentiable almost everywhere, then our construction and estimates, suitably modified, still apply.
\begin{remark}
Similar examples were already constructed in~\cite{dFHazT3}.
However, the construction there made determining the possible conjugacy between different $f_a$ difficult.
Our approach here simplifies this, while also giving the additional dynamical information in Theorem~\ref{mainthm:HolderSobolev} above.
\end{remark}
Following this we also construct examples of endomorphisms with infinite topological entropy satisfying the stronger Zygmund condition.
Namely, the following is shown.
\begin{maintheorem}\label{mainthm:Zygmund}
There exists $f\in C([0,1],[0,1])$ with the following properties
\begin{enumerate}
\item
\begin{enumerate}
\item
$f$ is not topologically conjugate to the examples in Theorem~\ref{mainthm:HolderSobolev}
\item
$h_\mathrm{top}(f)=+\infty$
\item
$f$ is not expansive, $h$-expansive or asymptotically $h$-expansive
\end{enumerate}
\item
$f$ satisfies the little Zygmund condition
\end{enumerate}
\end{maintheorem}
\begin{remark}
Observe that Theorem~\ref{mainthm:Zygmund} gives an affimative answer to Benedicks' question stated above.
\end{remark}
%
%%%%%%%%%%%%%%%
\begin{comment}
Since the space $C^Z([0,1],[0,1])$, when endowed with the Zygmund norm, has the property that 
%being a homeomorphism is an open property,and since homeomorphisms of $[0,1]$ have zero topological entropy 
an open neighbourhood of the identity consists of maps with zero topological entropy,
we also get the following.
%
\begin{maintheorem}\label{cor:zygmund-hinft-notgeneric}
Namely, in $C^Z([0,1],[0,1])$, endowed with the Zygmund norm, infinite topological entropy is not a generic property (it is not even a dense property),
but it does contain mappings with infinite topological entropy.  
\end{maintheorem}
%
\begin{remark}
This gives an answer to C.\ Tresser's question above.
\end{remark}
\end{comment}
%%%%%%%%%%%%%
%
%%%%%%%%%%%%%%%%%%%%%%%%%%%%%%%%%%%%
\subsection{Structure of the paper.}
%%%%%%%%%%%%%%%%%%%%%%%%%%%%%%%%%%%%
In Section~\ref{subsect:not+term}, 
we set up notation and terminology for the rest of the paper, and recall some basic facts.
In Section~\ref{sect:Holder+Sobolev} we give a proof of Theorem~\ref{mainthm:HolderSobolev}.
First we construct the one-parameter family $f_a$ 
from which it will be clear that properties 1(a)--1(c) of Theorem~\ref{mainthm:HolderSobolev} hold for all parameters $a\in (0,1]$.
After this an elementary proof of properties 2(a)--2(c), {\it i.e.}, for $a=1$, of Theorem~\ref{mainthm:HolderSobolev} is given.
Following this we prove some auxiliary propositions that are then used to prove properties 3(a)--3(b).
In Section~\ref{sect:Zygmund}, after recalling some basic definitions we give a proof of Theorem~\ref{mainthm:Zygmund}.
Finally, in Section~\ref{sect:conclusion} we end with some remarks and open problems.

%%%%%%%%%%%%%%%%%%%%%%%%%%%%%%%%%%%%%%
\subsection{Notation and terminology.}\label{subsect:not+term}
%%%%%%%%%%%%%%%%%%%%%%%%%%%%%%%%%%%%%%    
Throughout this article, we use the following notation.
We denote the Euclidean norm in $\mathbb{R}$ by $|\cdot|_{\mathbb{R}}$ or $|\cdot|$ when there is not risk of ambiguity.
We denote the Euclidean distance by $d(\cdot,\cdot)$.

%%%%%%%%%%%%%%%%%%%%%%%%%%%%%%%%%%
\subsubsection{H\"older Mappings.}
%%%%%%%%%%%%%%%%%%%%%%%%%%%%%%%%%%
Let $\alpha\in(0,1)$.
Given a subset $\Omega$ of $\mathbb{R}$, let $C^\alpha(\Omega,\mathbb{R})$ denote 
the set of real-valued functions $f$ on $\Omega$ 
satisfying
the {\it $\alpha$-H\"older condition}
\begin{equation}\label{cond:Holder_condition}
[f]_{\alpha,\Omega}
\;\stackrel{\tiny{\mathrm{def}}}{=}\;
\sup_{x,y\in \Omega; x\neq y} 
\frac{d(f(x),f(y))}{d(x,y)^\alpha}
<\infty \ .
\end{equation}
When the domain of $f$ is clear we will write $[f]_{\alpha}$ instead of $[f]_{\alpha,\Omega}$.
The set $C^\alpha(\Omega,\mathbb{R})$ has a linear structure and
$[\,\cdot\,]_{\alpha,\Omega}$ 
defines a semi-norm\footnote{This also induces a pseudo-distance which we will call the {\it $C^\alpha$-pseudo-distance}.}, 
which we call the {\it $C^\alpha$-semi-norm}.
Consequently
\begin{equation}
\|f\|_{C^\alpha(\Omega,\mathbb{R})}
\;\stackrel{\tiny{\mathrm{def}}}{=}\;
\|f\|_{C^0(\Omega,\mathbb{R})}+[f]_{\alpha,\Omega}
\end{equation}
defines a complete norm on 
$C^\alpha(\Omega,\mathbb{R})$. 

For the same subset $\Omega$ of $\mathbb{R}$ let $C^\mathrm{Lip}(\Omega,\mathbb{R})$ denote the set of 
real-valued functions $f$ on $\Omega$ satisfying the {\it Lipschitz condition}
\begin{equation}
[f]_{\mathrm{lip},\Omega}
\;\stackrel{\tiny{\mathrm{def}}}{=}\;
\sup_{x,y\in \Omega; x\neq y} 
\frac{d(f(x),f(y))}{d(x,y)}
<\infty \ .
\end{equation}
As before, this defines a semi-norm on the linear space $C^{\mathrm{Lip}}(\Omega,\mathbb{R})$.
We denote the corresponding norm by $\|\,\cdot\,\|_{C^\mathrm{Lip}(\Omega,\mathbb{R})}$.
Then, as above, this defines a Banach space structure on $C^\mathrm{Lip}(\Omega,\mathbb{R})$.

%%%%%%%%%%%%%%%%%%%%%%%%%%%%%%%%%
\subsubsection{Sobolev Mappings.}
%%%%%%%%%%%%%%%%%%%%%%%%%%%%%%%%%
Given an open subset $\Omega$ of $\mathbb{R}$, the Sobolev class $W^{1,p}(\Omega)$ consists of 
measurable functions 
$f\colon \Omega\to\mathbb{R}$
for which the first distributional partial derivative is defined and belongs to $L^p(\Omega)$.
Then $W^{1,p}(\Omega)$ is a Banach space with respect to the norm
\begin{equation}
\|u\|_{1,p}=\|u\|_{L^p}+\|Du\|_{L^p}
\end{equation}
Define the space
\begin{equation}
\mathbb{W}^{1,p}(\Omega,\mathbb{R})
=
W^{1,p}\left(\Omega,\mathbb{R}\right)\cap C^0\left(\overline{\Omega},\mathbb{R}\right)
\end{equation}
For $f\in\mathbb{W}^{1,p}(\Omega,\mathbb{R})$ define 
\begin{equation}
[f]_{W^{1,p},\Omega}
=\left(\int_\Omega|Df(x)|^p\,dx\right)^{\frac{1}{p}}
\end{equation}
Observe that 
$\mathbb{W}^{1,p}(\Omega,\mathbb{R})$ is a linear space and
that $[\,\cdot\,]_{W^{1,p},\Omega}$ defines a semi-norm which we call the {\it $W^{1,p}$-semi-norm}.
Setting
\begin{equation}
\|f\|_{\mathbb{W}^{1,p}(\Omega,\mathbb{R})}
=
\|f\|_{C^0(\overline\Omega,\mathbb{R})}
+[f]_{W^{1,p},\Omega}
\end{equation}
this defines a norm on 
$\mathbb{W}^{1,p}(\Omega,\mathbb{R})$ which is complete, and thus
$\mathbb{W}^{1,p}(\Omega,\mathbb{R})$ is endowed with the structure of a Banach space.

%%%%%%%%%%%%%%%%%%%%%%%%%%%%%%%%%%%%%%%%%%%%%%%%%%%%
\subsubsection{Topological Entropy and Expansivity.}
%%%%%%%%%%%%%%%%%%%%%%%%%%%%%%%%%%%%%%%%%%%%%%%%%%%%
Let $(X,d)$ be a compact metric space.
Let $f$ be a continuous self-map of $(X,d)$.
For each $n\in\mathbb{N}$ define the distance function
\begin{equation}
d_n^f(x,y)=\max_{0\leq k<n} d(f^k(x),f^k(y))
\end{equation}
%%%%%%%%%%%%%%
\begin{comment}
A set $E\subset X$ is {\it $(n,\delta)$-separated} with respect to $f$ if for any $x,y\in E$, $x\neq y$,
\begin{equation}
\max_{0\leq k<n}d(f^k(x),f^k(y))>\delta
\end{equation}
Define
\begin{equation}
s_f(n,\delta)
&=\max \left\{ \card E \ : \ E \ \mbox{is} \ (n,\delta)\mbox{-separated for} \ f\right\} 
\end{equation}
\end{comment}
%%%%%%%%%%%%%
Given sets $E,F\subset X$, we say that the set $E$ {\it $(n,\delta)$-spans} the set $F$ with respect to $f$ 
if for any $x\in F$, there exists $y\in E$ such that $d_n^f(x,y)<\delta$.
Let 
\begin{align}
r_f(n,\delta;F)
&=\min \bigl\{ \card E \ : \ E \ (n,\delta)\mbox{-spans} \ F \ \mbox{with respect to} \ f\bigr\}
\end{align}
(Note that: 
(i) if $F$ is compact then $r_f(n,\delta;F)<\infty$;
(ii) $r_f(n,\delta;F)$ increases as $\delta$ decreases.)
For each compact set $K\subset X$,
define\footnote{Here we depart from the notation originally due to Bowen~\cite{Bowen1972}.}
\begin{equation}
r_f(\delta;K)
=\limsup_{n\to\infty}\frac{1}{n}\log r_f(n,\delta;K)
\end{equation}
and
\begin{equation}
h(f;K)=\lim_{\delta\to 0} r_f(\delta;K)
\end{equation}
Since $X$ is compact we can define
\begin{equation}
h(f,\delta)=h(f,\delta;X)
\end{equation}
The {\it topological entropy} of $f$ is defined by
\begin{equation}
h_\mathrm{top}(f)=\lim_{\delta\to 0}h(f,\delta)=\sup_K h(f;K)
\end{equation}
For each $\epsilon>0$ and $x\in X$ define
\begin{equation}
\Phi_\epsilon(x)
=
\bigcap_{n\geq 0} f^{-n}B_\epsilon(f^n(x))
=
\left\{y : d\left(f^n(x),f^n(y)\right)\leq \epsilon, \ \forall n\geq 0\right\}
\end{equation}

Recall that $f$ is {\it expansive} if there exists $\epsilon>0$ with the following property: 
given any $x,y\in X$,
if $d(f^k(x),f^k(y))<\epsilon$, for all $k\in\mathbb{N}$, then $x=y$.
Define
\begin{equation}
h_f^*(\epsilon)=\sup_{x\in X} h(f;\Phi_\epsilon(x))
\end{equation}
Then $f$ is 
{\it $h$-expansive} if $h_f^*(\epsilon)=0$ for some $\epsilon>0$; and is 
{\it asymptotically $h$-expansive} if $\lim_{\epsilon\to 0}h_f^*(\epsilon)=0$.
%
%
%
%
%
%
%
%
%
%
%
%%%%%%%%%%%%%%%%%%%%%%%%%%%%%%%%%%%%%%%%%%%%%%%%%%%%%%%%%%%%%%%%%%%%%%%%%%%%%%%%%%%%%%%%%%%%
%%%%%%%%%%%%%%%%%%%%%%%%%%%%%%%%%%%%%%%%%%%%%%%%%%%%%%%%%%%%%%%%%%%%%%%%%%%%%%%%%%%%%%%%%%%%
%%%%%%%%%%%%%%%%%%%%%%%%%%%%%%%%%%%%%%%%%%%%%%%%%%%%%%%%%%%%%%%%%%%%%%%%%%%%%%%%%%%%%%%%%%%%
%%%%%%%%%%%%%%%%%%%%%%%%%%%%%%%%%%%%%%%%%%%%%%%%%%%%%%%%%%%%%%%%%%%%%%%%%%%%%%%%%%%%%%%%%%%%
\section{Examples in H\"older and Sobolev classes.}\label{sect:Holder+Sobolev}
We construct a family of endomorphisms $f_a$ 
of the unit interval, depending upon the parameter 
$a\in (0,1]$, such that each $f_a$ satisfies 
$h_{\mathrm {top}}(f_a)=\infty$, it is not expansive, $h$-expansive or even asymptotically $h$-expansive,
and such that all the $f_a$ are topologically conjugate.
The main part of the work will then be in showing each $f_a$ has some intermediate regularity between $C^0$ and Lipschitz.  
\begin{remark}
On compact one-manifolds, a theorem of Misiurewicz~\cite{Misiurewicz1978} states that positive topological entropy,
and thus infinite topological entropy, must come from some iterate possessing a horseshoe.
More precisely, if $h_\mathrm{top}(f)>0$, then 
there exist sequences of positive integers $k_n$ and $s_n$ such that, for each $n$,
$f^{k_n}$ possesses an $s_n$-branched horseshoe\footnote{
A map $g$ possesses an {\it $s$-branched horseshoe} if there is an interval $J$ with $s$ pairwise disjoint subintervals $J_1,J_2,\ldots,J_s$,
such that $g(J_j)\subseteq J$ for $j=1,2,\ldots,s$.}
and
\begin{equation}
\lim_{n\to\infty}\frac{1}{k_n}\log s_n=h_\mathrm{top}(f)
\end{equation}
Thus examples given below, which are constructed so that certain iterates possess horseshoes, 
are somehow indicative of the general case.
\end{remark}
It will be useful to first consider an auxiliary 
family $g_{a,b}$ of interval maps defined as follows.  
First fix a positive integer $b$.  
Given an arbitrary interval $J$, 
let $A_J$ denote the unique orientation-preserving affine bijection from $J$ to $[0,1]$. 
Subdivide the interval $[0,1]$ into $b$ closed intervals 
$
J_{b,0}, J_{b,1},\ldots,J_{b,b-1}
$
of equal length, ordered from left to right. 
Let 
$A_{b,k}=A_{J_{b,k}}$ for each $k=0,1,\ldots,b-1$. 
Let 
$\nu$ denote the unique orientation-reversing affine bijection of $[0,1]$ to itself, 
For each $k=0,1,\ldots,b-1$, define
\begin{equation}
g_{1,b}(x) = \nu^k\circ A_{b,k}(x), \qquad \forall x\in J_{b,k}\,.
\end{equation}
More explicitly
\begin{equation}\label{eq:g1b-explicit}
g_{1,b}(x)
=\left\{\begin{array}{ll}
bx-k     & x\in J_{b,k}, \ \ k \ \mbox{even} \\
(k+1)-bx & x\in J_{b,k}, \ \ k \ \mbox{odd}
\end{array}\right.
\end{equation}
Observe that $g_{1,b}$ is continuous on $[0,1]$.
Also, $[0,1]$ possesses a $g_{1,b}$-invariant subset on which $g_{1,b}$ it acts as the unilateral shift on $b$ symbols. 
In fact, $h_{\mathrm{top}}(g_{1,b})=\log b$
(see e.g.~\cite[Section 3.2.c]{KandH}).

Next, take a continuous one-parameter family $\varphi_a$, $a\in (0,1]$, of orientation-preserving homeomorphisms of $[0,1]$, 
with $\varphi_1=\mathrm{id}$, and define
\begin{equation}
g_{a,b}=\varphi_a\circ g_{1,b}\circ \varphi_a^{-1}
\end{equation}
For example, we could take $\varphi_a$ equal to $q_a(x)=x^a$, the power function of exponent $a$.
(Observe that in this case $g_{a,b}$ is $C^a$ but not $C^\alpha$ for any $\alpha>a$, provided that $b\geq 2$.)
Then $g_{a,b}$ is continuous on $[0,1]$.
As topological entropy is invariant under topological conjugacy, we also have 
$h_{\mathrm{top}}(g_{a,b})=\log b$, for each $a\in (0,1]$ and each positive integer $b$.
We call $b$ the {\it number of branches} of $g_{a,b}$ and $a$ the {\it order of singularity}.  

We now define the family $f_a$ as follows.  
For each positive integer $n$ define the interval $I_n=(2^{-n},2^{-n+1}]$ and let $f_a$ be given by
\begin{equation}
f_a(x)=
\left\{\begin{array}{ll}
A_{I_n}^{-1}\circ g_{a,2n+1}\circ A_{I_n}(x) & x\in I_n, \ n=1,2\ldots \\
0 & x=0
\end{array}\right.
\end{equation}
Observe that, since $g_{a,2n+1}$ fixes the endpoints of $[0,1]$ and is continuous, the map $f_a$ is also continuous.
Also, since, for each fixed $b$, all the functions $g_{a,b}$, $a\in (0,1]$ are topologically conjugate, 
it follows that all the functions $f_a$, $a\in (0,1]$, are also 
all topologically conjugate. Namely, $f_a=\psi_a^{-1}\circ f_1\circ \psi_a$ where
\begin{equation}
\psi_a(x)=A_{I_n}^{-1}\circ \varphi_a\circ A_{I_n}(x) \qquad \forall x\in I_n, \ \forall n\in\mathbb{N}
\end{equation}
Notice that the closure of each interval $I_n$ is totally invariant.  
Since the 
topological entropy of a map is the supremum of the topological entropy of its restriction to all closed invariant subsets,
since topological entropy is invariant under topological conjugacy (see e.g.~\cite[Section 3.1.b]{KandH}) and, as was stated above, 
$h_{\mathrm{top}}(g_{a,b})=\log b$ for all $b$,   
it follows that
\begin{equation}\label{eq:infinite_entropy-fabI}
h_{\mathrm{top}}(f_a) 
\;\geq\; 
\sup_n h_{\mathrm{top}}(f_a|_{I_n}) 
\;=\;
\sup_n h_{\mathrm{top}}(g_{a,2n+1})
=+\infty \ .
\end{equation}
Next, observe that, as $f_a$ has arbitrarily small invariant subsets (namely the intervals $I_n$) 
the function $f_a$ cannot be expansive.
In fact, since 
$h(f_a,I_n)=\log (2n+1)$ for each $n$, it follows that 
\begin{equation}
\lim_{\epsilon\to 0}h^*_{f_a}(\epsilon)
\;\geq\;
\lim_{n\to\infty} h(f_a;I_n)
\;=\;+\infty
\end{equation}
Thus $f_a$ is neither $h$-expansive nor asymptotically $h$-expansive.
Therefore properties 1(a)--1(c) of Theorem~\ref{mainthm:HolderSobolev} hold.
\begin{remark}
That $f_a$ is not asymptotically $h$-expansive could 
also be shown using topological conditional entropy 
in the following way.
By~\cite[Proposition 3.3]{Misiurewicz1976} 
infinite topological entropy $h_\mathrm{top}(f_a)$ 
implies infinite topological conditional entropy $h^*(f_a)$.
However, by~\cite[Corollary 2.1(b)]{Misiurewicz1976} $f_a$ is asymptotically $h$-expansive if and only if $h^*(f_a)=0$. 
\end{remark}
\begin{figure}[htp]
\begin{center}
\psfrag{0}[][]{ $0$} 
\psfrag{x}[][][1]{$x$}
\psfrag{f}[][][1]{$\ \ f(x)$} 

%\psfrag{a}[][]{$f^{q_{n}-q_{n-1}}(c_{0})$} 
%\psfrag{b}[][]{$f^{q_{n}}(c_{0})$} 
%\psfrag{c}[][]{$c_{0}$} 
%\psfrag{d}[][]{$f^{-q_{n-1}}(c_{0})$} 
%\psfrag{e}[][]{$f^{2q_{n-1}}(c_{0})$} 
\includegraphics[width=3.5in]{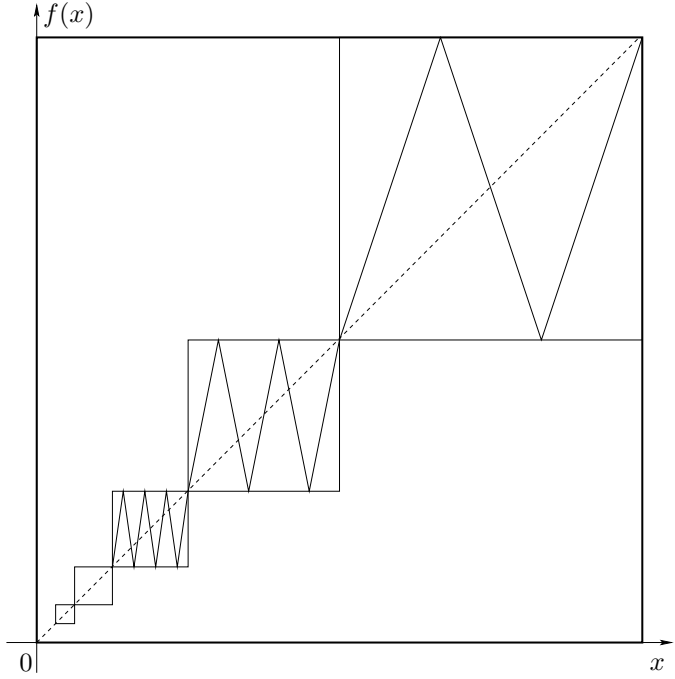}
\end{center}
\caption[slope]{\label{figendo} The graph of a H\"older interval endomorphism with infinite topological entropy.}
\end{figure}
\begin{proof}[Proof of Theorem~\ref{mainthm:HolderSobolev} 2(a)--2(c)]
%%%%%%%%%%%%%%%%%%MODULUS OF CONT.%%%%%%%%%%%%%%%%%%%%
%%%%%%%%%%%%%%%%%%%%%%%%%%%%%%%%%%%%%%%%%%%%%%%%%%%%%%
For each positive integer $n$, 
define the subintervals 
$I_{n,k}=A_{I_n}^{-1}(J_{2n+1,k})$ of $I_n$ for $k=0,1,\ldots,2n$.
These denote the maximal closed subintervals of $I_n$ on which $f$ is affine.
%({\it I.e.\/}, the branches of $f|I_k$.)

\vspace{5pt}

\noindent
(a) 
Take distinct points $x,y\in [0,1]$.
There are three cases to consider.
\begin{description}
\item[($x\in I_n, y\in I_m$, $n>m$)]
Since $I_m$ and $I_n$ are both $f$-invariant, 
$f(x)\in I_n$ and $f(y)\in I_m$.
Moreover, observe that
\begin{equation}
|f(x)-f(y)|\leq |2^{-n}-2^{-m+1}|<2^{-m+1}
\end{equation}
together with
\begin{equation}
|x-y|\geq |2^{-n+1}-2^{-m}|\geq 2^{-m-1}
\end{equation}
implies that
\begin{equation}
\frac{|f(x)-f(y)|}{\omega(|x-y|)}
\leq \frac{2^{-m+1}}{2^{-m-1}\log 2^{m+1}}
=\frac{4}{(m+1)\log 2} \ .
\end{equation}
\item[($x=0, y\in I_m$)]
Applying the same argument as in the previous case and observing that $f(x)=x=0$ we find that
\begin{equation}
\frac{|f(x)-f(y)|}{\omega(|x-y|)}
%=\frac{|f(y)|}{y\log (1/y)}
\leq \frac{2^{-m+1}}{2^{-m-1}\log 2^{m+1}}
= \frac{4}{(m+1)\log 2}
\leq \frac{2}{\log 2} \ .
\end{equation}

\item[($x\in I_n, y\in I_m$, $n=m$)]
If $x$ and $y$ do not lie in the same branch of $f|_{I_k}$, 
then there exists $y'$, in the same branch as $x$, 
satisfying $f(y)=f(y')$ and $|x-y|>|x-y'|$.
Moreover,
\begin{equation}
|x-y'|
\leq \frac{|I_m|}{2m+1}
=\frac{1}{2^m(2m+1)} \ .
\end{equation}
Thus
\begin{align}
\frac{|f(x)-f(y)|}{\omega(|x-y|)}
\leq\frac{|f(x)-f(y')|}{\omega(|x-y'|)}
&=\frac{(2m+1)|x-y'|}{|x-y'|\log (|x-y'|^{-1})}\\
&\leq \frac{2m+1}{\log 2^m (2m+1)}\\
%&=\frac{2m+1}{n\log 2+\log(2m+1)}
&\leq \frac{2}{\log 2}+1 \ .
\end{align}
\end{description} 
In each of these cases, for $x,y\in [0,1], x\neq y$,
\begin{equation}
\frac{|f(x)-f(y)|}{\omega(|x-y|)}\leq \frac{2}{\log 2}+1
\end{equation}
and hence $f$ has modulus of continuity $\omega$, which completes the proof of part (i).

%%%%%%%%%%%%%%%%%%ABSOLUTE CONTINUITY%%%%%%%%%%%%%%%%%
%%%%%%%%%%%%%%%%%%%%%%%%%%%%%%%%%%%%%%%%%%%%%%%%%%%%%%
\begin{comment}
Next, observe that $f$ has variation
\begin{equation*}
\mathrm{Var}_{[0,1]}(f)
=\sum_{k=1}^\infty \mathrm{Var}_{[0,1]}(f|I_k)
=\sum_{k=1}^\infty (2k+1)|I_k|
=\sum_{k=1}^\infty \frac{(2k+1)^2}{e^{Lk^2(2k+1)^2}} \ ,
\end{equation*}
which is finite.
Thus $f$ has bounded variation. 
Moreover $f$ satisfies Lusin's $N$-property, and it is clearly continuous.
Therefore, by~\cite[Theorem 4, Ch. IX, Section 3]{NatansonBook}, 
this implies that $f$ is absolutely continuous. 
\end{comment}
%%%%%%%%%%%%%%%%%%%%%%%%%%%%%%%%%%%%%%%%%%%%%%%%%%%%%%
%%%%%%%%%%%%%%%%%%%%%%%%%%%%%%%%%%%%%%%%%%%%%%%%%%%%%%

%%%%%%%%%%%%%%%%%%%%%%W^{1,p}%%%%%%%%%%%%%%%%%%%%%%%%%
%%%%%%%%%%%%%%%%%%%%%%%%%%%%%%%%%%%%%%%%%%%%%%%%%%%%%%
\vspace{5pt}

\noindent
(b) Observe that $f$ is differentiable except at the endpoints of the intervals $I_{k,l}$.
Hence
\begin{equation}
\left|f'|_{I_n}\right|=\frac{|I_n|}{|I_{n,k}|}=2n+1 \ .
\end{equation}
Therefore, as the $I_n$ form a measurable partition of $[0,1]$,
\begin{align}
\int_{[0,1]} |f'|^p\,dx
=\sum_{n=1}^\infty \int_{I_n}|f'|^p\,dx
&=\sum_{n=1}^\infty (2n+1)^p\int_{I_n}\,dx\\
%&=\sum_{n=1}^\infty (2n+1)^p |I_k|\\
&=\sum_{n=1}^\infty (2n+1)^p 2^{-k}\\
&\leq \sum_{n=1}^\infty n^p 2^{-(n-1)/2}\\
&= 2^{1/2}\sum_{n=1}^\infty n^p 2^{-n/2} \ .
\end{align}
However, this last series is finite. This follows, for example, 
since $n^p 2^{-n/2}<n^{-2}$ for all $n$ sufficiently large 
and by recalling that 
$\sum_{n=1}^\infty n^{-2}<\infty$.
Consequently the Sobolev norm of $f$ is finite and hence $f\in W^{1,p}([0,1])$.
%%%%%%%%%%%%%%%%%%%%%%%%%%%%%%%%%%%%%%%%%%%%%%%%%%%%%%
%%%%%%%%%%%%%%%%%%%%%%%%%%%%%%%%%%%%%%%%%%%%%%%%%%%%%%

\vspace{5pt}

\noindent
(c) First note that as $g_{1,b}$ preserves Lebesgue measure $\mu$ for each $b$, it follows trivially that Lebesgue measure is invariant under $f_1$.
Since $h_\mu(g_{1,b})=\log b$, it also follows that Lebesgue measure is a measure of maximal entropy.
Hence the theorem is shown.
\end{proof}
For each positive integer $n$, by performing the same construction as above but just on the union of the intervals 
$I_n, I_{n+1},\ldots$ 
we also get the following corollary.
\begin{corollary}\label{cor:a=1_accum_on_id}
There exists a sequence $f_n\in C^0([0,1],[0,1])$ satisfying properties 2(a)-2(c) in Theorem~\ref{mainthm:HolderSobolev} above and
with the additional property that $\lim_{n\to\infty} f_n= \mathrm{id}$ where convergence is taken
\begin{itemize}
\item
in the $C^\alpha$-topology for any $\alpha\in(0,1)$,
\item
in the $W^{1,p}$-topology for any $p\in[1,\infty)$.
\end{itemize}
\end{corollary}
We recall that maps with modulus of continuity $t\log(1/t)$ are in the H\"older class 
$C^\alpha$ for every $\alpha\in [0,1)$, but they are not necessarily Lipschitz.  
Moreover, the map $f_1$ is a $C^\alpha$-limit 
of piecewise-affine maps. Hence it lies in the $C^\alpha$-boundary of the space of 
Lipschitz maps.  
When $a\neq 1$, the map $f_a$ does not satisfy this property.  
%%%%%%%%%%%%%%%
\begin{comment}
More precisely, we have the following result.
%
\begin{theorem}\label{thm:eg2}
For $f=f_a$, $a\in (0,1)$, 
the following holds.
\begin{itemize}
\item[(i)]
$f$ 
is $C^\alpha$ if and only if $\alpha\leq a \,.$
\item[(ii)]
$f$ 
is $W^{1,p}$ if and only if $p<(1-a)^{-1} \,.$ 
\item[(iii)]
$h_{\mathrm{top}}(f)=+\infty\,. $
\end{itemize}
\end{theorem}
\end{comment}
%%%%%%%%%%%%%%
%
The proof of Theorem~\ref{mainthm:HolderSobolev} 3(a)--3(b) could be made 
using the argument presented above in the proof of
properties 2(a)-2(b) of Theorem~\ref{mainthm:HolderSobolev}. 
However, we give a different proof below. 
For that we need the following gluing principle, which is essentially an application of Jensen's inequality.
\newpage
\begin{proposition}[Gluing Principle]\label{prop:gluing}
Let $\omega$ be a continuous, monotone-increasing function, locally concave at $\omega(0)=0$. %e.g. $\omega(x)=x\log\left(\frac{1}{x}\right)$.
Let $f$ be a continuous self-mapping of the compact interval $I$.
Let 
$I_1,I_2,\ldots$ denote a collection of closed intervals with pairwise disjoint interiors, covering $I$, and 
with the property that $f|_{I_k}$ has modulus of continuity $\omega$, for all $k$.
Let $C_k$ denote the $\omega$-semi-norm of $f|_{I_k}$.
If
\begin{itemize}
\item[(i)]
$\sum_{k=1}^\infty C_k<\infty$ then $f$ has modulus of continuity $\omega$ with 
$\omega$-semi-norm bounded by $C=\sum_{k=1}^\infty C_k$.  
%%%%%%%%%%%%%%%
\begin{comment}
\item[(i')]
If each $f|_{I_k}$ is onto $I$ for each $k$ and $\sup_k C_k<\infty$ 
then $f$ has modulus of continuity $\omega$ with $\omega$-semi-norm bounded by 
$C=\sup_k C_k$.
\end{comment}
%%%%%%%%%%%%%
\item[(ii)]
$\sup_k C_k<\infty$
and $f|_{\partial I_k}=\mathrm{id}$ for all $k$, 
then $f$ has modulus of continuity $\omega$ with $\omega$-semi-norm
bounded by $C=\frac{\mathrm{diam}(I)}{\omega(\mathrm{diam}(I))}+2\sup_k C_k$.
\end{itemize}
\end{proposition}
%
%%%%%%%%%%%%%%%%%%%%%%%%%%%%%%%%%%%%%%%%%%%%%
%\begin{comment}
%%%%%%%%%%%%%%%%%%%%%%%%%%%%%%%%%%%%%%%%%%%%%
\begin{proof}
For notational simplicity, assume that the intervals $I_k$ are ordered from left to right.
This does not affect the proof, but simplifies indexing.

\vspace{5 pt}

\noindent
Case (i).
Take $x,y\in I$. 
Assume that $x<y$.
Then there exist integers $m<n$ such that $x\in I_m, y\in I_n$.
Consequently 
\begin{equation}
x_m=x< x_{m+1}<\ldots<x_{n}<y= x_{n+1} \ ,
\end{equation}
where the points
$x_{m+1},x_{m+2}\ldots,x_{n}$ 
denote the left endpoints of the respective intervals 
$I_{m+1}, I_{m+2},\ldots, I_{n}$. 
Let $C_k$ denote the $\omega$-semi-norm of $f|_{I_k}$,
that is
\begin{equation}
C_k
=
\sup_{z\neq w\in I_k} \frac{d(f(z),f(w))}{\omega(d(z,w))} \ .
\end{equation}
%%%%%%%%%%%%FOLLOWING SEEMS UNNECESSARY%%%%%%%%%%%%
%Let $K_k$ denote the Lipschitz constant of $f|I_k$.
%Then the Lipschitz condition implies that for 
%$z\neq w\in I_k$,
%\begin{equation}
%d(f(z),f(w))
%%=K_k d(z,w)
%=K_k\left\{\log\left(\frac{1}{d(z,w)}\right)\right\}^{-1}\omega(d(z,w))
%\end{equation}
%so that
%\begin{equation}
%C_k
%=\sup_{z\neq w\in I_k} K_k\left\{\log\left(\frac{1}{d(z,w)}\right)\right\}^{-1}
%=K_k\left\{\log\left(\frac{1}{|I_k|}\right)\right\}^{-1}
%\end{equation}
%%%%%%%%%%%%%%%%%%%%%%%%%%%%%%%%%%%%%%%%%%%%%%%%%%%
It follows that
\begin{align}
d(f(x),f(y))
\leq\sum_{k=m}^n d(f(x_k),f(x_{k+1}))
\leq\sum_{k=m}^n C_k \omega\left(d(x_k,x_{k+1})\right) \ . \label{ineq:1}
\end{align}
However, since $\Lambda$ is concave, Jensen's inequality implies that
\begin{align}
\frac{\sum_{k=m}^nC_k\omega\left(d(x_k,x_{k+1})\right)}
{\sum_{k=m}^n C_k}
&\leq 
\omega\left(\frac{\sum_{k=m}^{n}C_k d(x_k,x_{k+1})}{\sum_{k=m}^n C_k}\right)\\
&\leq 
\omega \left(\max_{m\leq k\leq n} C_k \cdot \frac{\sum_{k=m}^{n} d(x_k,x_{k+1})}{\max_{m\leq k\leq n}C_k}\right) \\
&=
\omega \left(d(x,y)\right) \ . \label{ineq:2}
\end{align}
where, for the last equality we have used that the points $x_k$ are in the real line, placed in increasing order.
Combining inequalities~\eqref{ineq:1} with~\eqref{ineq:2} together with the hypothesis that $\sum_{k=1}^\infty C_k<\infty$ 
gives the result by taking the supremum over all possible $x$ and $y$.

\vspace{5 pt}

\noindent
Case (ii).
Take $x$, $y$ and $x_{m+1},\ldots,x_n$ as before.
Then
\begin{align}
d(f(x),f(y))
&\leq 
d(f(x),f(x_{m+1}))+d(f(x_{m+1}),f(x_{n}))+d(f(x_{n}),f(y))\\
&\leq 
C_m \omega(d(x,x_{m+1}))+d(x_{m+1},x_{n})+C_n \omega(d(x_{n},y))\\
&\leq
\left(2\sup_k C_k +\frac{\mathrm{diam}(I)}{\omega(\mathrm{diam}(I))}\right)\omega(d(x,y)) \ .
\end{align}
As this holds for all $x$ and $y$, it follows that $f$ has modulus of continuity $\omega$, 
with $\omega$-semi-norm bounded by 
$2\sup_k C_k +\mathrm{diam}(I)/\omega(\mathrm{diam}(I))$, as required.
\end{proof}
%%%%%%%%%%%%%%%%%%%%%%%%%%%%%%%%%%%%%%%%%%%%%%
%\end{comment}
%%%%%%%%%%%%%%%%%%%%%%%%%%%%%%%%%%%%%%%%%%%%%%
%
We will also need the following estimates for the map $g_{a,b}$ with respect to the H\"older and Sobolev semi-norms, in the case when $\varphi_a$ is a general concave 
orientation-preserving homeomorphism.
\begin{lemma}[Auxiliary Lemma]
Let $g_{a,b}$ be defined as above, 
where $\varphi_a$ is an arbitrary concave, orientation-preserving homeomorphism, 
so that it possesses an extension to $[0,1+\frac{1}{b}]$, which is also concave and a homeomorphism onto its image.
Then
\begin{equation}\label{ineq:aux-Calpha}
[g_{a,b}]_{C^\alpha,[0,1]}
\leq 
[\varphi_a]_{C^\alpha,[0,1]} 
\cdot b^{\alpha+1}
\int_{[\frac{1}{b},1+\frac{1}{b}]}\left|(\varphi_a^{-1})'(t)\right|^{\alpha}dt
\end{equation}
and
\begin{equation}\label{ineq:aux-W1p}
[g_{a,b}]_{W^{1,p},[0,1]}^p
\leq 
[\varphi_a]_{W^{1,p},[0,1]}
\cdot b^p 
\left(\int_{[0,1]}\left|\frac{t}{g_{1,b}(t)}\right|^{\frac{p^2(1-a)}{p-1}}dt\right)^{1-\frac{1}{p}} \ .
\end{equation}
\end{lemma}
\begin{remark}
As $\varphi_a$ is monotone increasing it follows by Lebesgue's Last Theorem that is is differentiable Lebesgue-almost everywhere 
(see, e.g.~\cite{PughBook}).
Since it is concave it follows from Alexandrov's theorem that is it also twice-differentiable Lebesgue-almost everywhere~\cite[Section 6.4]{EvansGariepy}.
\end{remark}
\begin{proof}
Before starting the proof, 
we introduce the following notation and make the following comments.
For any $t\in [0,1]$ we use the notation
\begin{equation}
t'=\varphi_a^{-1}(t) \ , \qquad
t''=g_{1,b}(\varphi_a^{-1}(t)) \ .
\end{equation}
%%%%%%%%%%%%%%%%%%%
First consider the H\"older estimate.
Take $k\in \{0,1,\ldots,b-1\}$. 
Let $x,y\in J_{b,k}$ be arbitrary distinct points.
Then, by telescoping the $a$-H\"older difference 
quotient, and observing that $g_{1,b}$ is affine, 
we find that
\begin{align}
\frac{|g_{a,b}(x)-g_{a,b}(y)|}{|x-y|^\alpha}
%&= 
%\frac{|\varphi_a(x'')-\varphi_a(y'')|}{|x''-y''|^a}
%\frac{|g_{1,b}(\varphi_a^{-1}(x))-g_{1,b}(\varphi_a^{-1}(y))|^a}{|x-y|^a}\\
&=
b^\alpha
\frac{|\varphi_a(x'')-\varphi_a(y'')|}{|x''-y''|^a}
\left(\frac{|\varphi_a^{-1}(x)-\varphi_a^{-1}(y)|}{|x-y|}\right)^\alpha \ .
\end{align}
Observe that $x''$ and $y''$ take values throughout $[0,1]$.
Therefore
\begin{equation}
\frac{|\varphi_a(x'')-\varphi_a(y'')|}{|x''-y''|^\alpha}\leq [\varphi_a]_{C^\alpha,[0,1]} \ .
\end{equation}
Next, trivially 
$x$ and $y$ take values throughout $J_{b,k}=[\frac{k}{b},\frac{k+1}{b}]$.
Therefore, since the function $\varphi_a^{-1}(t)$ is convex and increasing on the positive real line 
(and thus difference quotients on $J_{b,k}$ are maximised by the derivative at the right endpoint $\partial^+ J_{b,k}$),
\begin{equation}
%(\varphi_a^{-1})'\left(\frac{k}{b}\right)
%=
(\varphi_a^{-1})'(\partial^- J_{b,k})
\leq
\frac{|\varphi_a^{-1}(x)-\varphi_a^{-1}(y)|}{|x-y|}
\leq
(\varphi_a^{-1})'(\partial^+ J_{b,k}) \ .
%=
%(\varphi_a^{-1})'\left(\frac{k+1}{b}\right)
\end{equation}
Consequently, by Proposition~\ref{prop:gluing}(i), 
together with the fact that $(\varphi_a^{-1})'$ is increasing on the positive real line
(so $(\varphi_a^{-1})'$ is minimised on $J_{b,k}$ 
by its value at the left endpoint $\partial^- J_{b,k}=\frac{k}{b}$) we have
\begin{align}
[g_{a,b}]_{C^\alpha,[0,1]}
&\leq 
\sum_{k=0}^{b-1} [g_{a,b}]_{C^\alpha,J_{b,k}}\\
&\leq
[\varphi_a]_{C^\alpha,[0,1]}\cdot
b^{\alpha+1}
\sum_{k=0}^{b-1} \left|(\varphi_a^{-1})'\left(\tfrac{k+1}{b}\right)\right|^\alpha \cdot \tfrac{1}{b}\\
&=
[\varphi_a]_{C^\alpha,[0,1]}\cdot
b^{\alpha+1}
\sum_{k=0}^{b-1} \left|(\varphi_a^{-1})'\left(\partial^- J_{b,k+1}\right)\right|^\alpha \cdot |J_{b,k+1}|\\
&\leq
[\varphi_a]_{C^\alpha,[0,1]}\cdot
b^{\alpha+1} 
\int_{[\frac{1}{b},1+\frac{1}{b}]} \left|(\varphi_a^{-1})'(t)\right|^\alpha\,dt \ .
\end{align}
Next, consider the Sobolev case.
Observe that $g_{a,b}$ is differentiable everywhere except a finite set of points. 
More precisely, $g_{a,b}$ has breaks at exactly the endpoints of $\varphi_a(J_{b,k})$ for $k=0,1,\ldots,b-1$.
By the chain rule, at Lebesgue almost every point $x$ we have
\begin{align}
|g_{a,b}'(x)|
&=|\varphi_a'(g_{1,b}(\varphi_a^{-1}(x))| \, |g_{1,b}'(\varphi_a^{-1}(x))| \, |(\varphi_a^{-1})'(x)|\\
&=b\left(\frac{|g_{1,b}(\varphi_a^{-1}(x))|}{|\varphi_a^{-1}(x)|}\right)^{a-1} \ .
\end{align}
Thus, by the change of variable formula for integrals
%i.e. $\in_{\phi(I)}f(x)\,dx=\int_{I} f(\phi(t))\phi'(t)\,dt$
\begin{align}
[g_{a,b}]_{W^{1,p},\varphi_a(J_{b,k})}^p
&=
\int_{\varphi_a(J_{b,k})}
|g_{a,b}'(x)|^p 
\,d\mu(x)\\
&=
b^p\int_{\varphi_a(J_{b,k})}
\left|\frac{g_{1,b}(\varphi_a^{-1}(x))}{\varphi_a^{-1}(x)}\right|^{p(a-1)}  
\,d\mu(x)\\
&=
b^p\int_{J_{b,k}}
\left|\frac{g_{1,b}(t)}{t}\right|^{p(a-1)} 
|\varphi_a'(t)|
\,d\mu(t) \ .
\end{align}
Therefore
\begin{equation}
[g_{a,b}]_{W^{1,p},[0,1]}^p
\;=\;
\sum_{k=0}^{b-1}[g_{a,b}]_{W^{1,p},J_{b,k}}^p
\;=\;
b^p\int_{[0,1]}
\left|\frac{t}{g_{1,b}(t)}\right|^{p(1-a)} 
|\varphi_a'(t)|
\,d\mu(t) \ .
\end{equation}
Therefore applying H\"older's inequality gives the result.
\end{proof}
\begin{corollary}\label{cor:q_a-Holder+Sobolev_bounds}
Let $b\geq 2$.
Let $\varphi_a=q_a$ for each $a\in [0,1]$.
Then there exist positive real numbers $C(a,\alpha)$ and $K(a,p)$, depending only upon $a$ and $\alpha$ and upon $a$ and $p$ respectively,
such that we have the following for each $a\in [0,1)$:
\begin{enumerate}
\item[(i)]
$g_{a,b}$ is $C^\alpha$ for all $\alpha\leq a$ and 
\begin{equation*}
[g_{a,b}]_{C^\alpha,[0,1]}\;\leq\; C(a,\alpha) b^{\alpha+1}
\end{equation*}
\item[(ii)]
$g_{a,b}$ is $W^{1,p}$ for all $p< (1-a)^{-1}$ and 
\begin{equation*}
[g_{a,b}]_{W^{1,p},[0,1]}^p\;\leq\; K(a,p)b^{p(1-a)+1} \ .
\end{equation*}
\end{enumerate}
\end{corollary}
\begin{proof}
\noindent
(i)
When $\varphi_a=q_a$ we find that $q_{a,b}$ is locally of the form 
$|x|^{a}$ about any $g_{a,b}$-preimage of $0$ except $0$ itself. 
(Observe that, as $b\geq 2$, such a preimage exists.)
Therefore $g_{a,b}$ cannot be $C^\alpha$ for any $\alpha>a$.
For $\alpha\leq a$, by inequality~\eqref{ineq:aux-Calpha},
\begin{equation}
[g_{a,b}]_{C^\alpha,[0,1]}
\leq
[\varphi_a]_{C^\alpha,[0,1]}
\frac{b^{\alpha+1} a^{1-\alpha}}{\alpha(1-a)+a}
\left((1+\tfrac{1}{b})^{\alpha(\frac{1}{a}-1)+1}-(\tfrac{1}{b})^{\alpha(\frac{1}{a}-1)+1}\right) 
\end{equation}
and thus there exists an extended positive real number $C(a,\alpha)$, 
%$=[\varphi_a]_{C^\alpha,[0,1]}2^{\alpha(\frac{1}{a}-1)+1} a^{1-\alpha}/(\alpha(1-a)+1)$
depending upon $a$ and $\alpha$ only, such that 
$C(a,\alpha)$ 
is finite for $0\leq \alpha\leq a$,
and is infinite for $a<\alpha\leq 1$, and for which
\begin{equation}
[g_{a,b}]_{C^\alpha,[0,1]}
\leq
C(a,\alpha) b^{\alpha+1} \ .
\end{equation}

\noindent
(ii)
First, in the special case when $k=0$,
\begin{equation}
[g_{a,b}]_{W^{1,p},\varphi_a(J_{b,0})}^p
=b^p\int_{0}^{\frac{1}{b}}\left|\frac{1}{b}\right|^{p(1-a)}\varphi_a'(t)\,dt
%=b^{pa}(\varphi_a(\frac{1}{b})-\varphi_a(0))
=b^{a(p-1)} \ .
\end{equation}
In the general case, 
applying the standard $L^1$-estimate to inequality~\eqref{ineq:aux-W1p} gives 
\begin{equation}\label{ineq:W1p-L1ineq}
\begin{gathered}
b^p
\min_{J_{b,k}} |\varphi_a'|
\int_{J_{b,k}}
\left|\frac{t}{g_{1,b}(t)}\right|^{p(1-a)} 
\,dt 
\;\leq\;
[g_{a,b}]_{W^{1,p},\varphi_a(J_{b,k})}^p\\
\;\leq\; 
b^p
\max_{J_{b,k}} |\varphi_a'|
\int_{J_{b,k}}
\left|\frac{t}{g_{1,b}(t)}\right|^{p(1-a)} 
\,dt \ .
\end{gathered}
\end{equation}
But
\begin{align}
\int_{J_{b,k}}
\left|\frac{t}{g_{1,b}(t)}\right|^{p(1-a)} 
\,dt
&=
\left\{\begin{array}{ll}
\displaystyle{\int}_{J_{b,k}}
\left|\tfrac{t}{bt-k}\right|^{p(1-a)} 
\,dt
& \ \ k \ \mbox{even}\\
\displaystyle{\int}_{J_{b,k}}
\left|\tfrac{t}{k+1-bt}\right|^{p(1-a)} 
\,dt
& \ \ k \ \mbox{odd}
\end{array}\right. \ .
%&=
%\left\{\begin{array}{ll}
%\int_{[0,1]}
%\left|\tfrac{(u+k)/b}{u}\right|^{p(1-a)} 
%b\,du
%& \ \ k \ \mbox{even}\\
%\int_{[0,1]}
%\left|\tfrac{(k+1-u)/b}{u}\right|^{p(1-a)} 
%b\,du
%& \ \ k \ \mbox{odd}
%\end{array}\right.\\
\end{align}
Making an appropriate change of variables this can also be written in the form
\begin{align}
\int_{J_{b,k}}
\left|\frac{t}{g_{1,b}(t)}\right|^{p(1-a)} 
\,dt
&=
\left\{\begin{array}{ll}
\displaystyle{\int}_{[0,1]}
\left|1+\tfrac{k}{u}\right|^{p(1-a)} 
b^{1-p(1-a)}\,du
& \ k \ \mbox{even}\\
\displaystyle{\int}_{[0,1]}
\left|1-\tfrac{k+1}{u}\right|^{p(1-a)} 
b^{1-p(1-a)}\,du
& \ k \ \mbox{odd}
\end{array}\right. \ .
\end{align}
Combining with~\eqref{ineq:W1p-L1ineq}, 
this shows that $g_{a,b}$ is not $W^{1,p}$ for $p\geq (1-a)^{-1}$.
For $p<(1-a)^{-1}$, since the power function 
$t^{\sigma}$, where $\sigma=p(1-a)$, 
is concave we find that
%\begin{equation}
%\left(1+\frac{k}{u}\right)^{\sigma}
%\leq 
%\left(\frac{k}{u}\right)^{\sigma}+\sigma\left(\frac{k}{u}\right)^{\sigma-1}
%\end{equation}
%it follows that
\begin{align}
\int_{[0,1]}
\left(1+\frac{k}{u}\right)^{\sigma}\,du
&\leq 
k^{\sigma}
\int_{[0,1]}
u^{-\sigma}\,du
+
\sigma k^{\sigma-1}
\int_{[0,1]}u^{1-\sigma}\,du\\
%&=
%k^{\sigma}
%\frac{u^{1-\sigma}}{1-\sigma}|_0^1
%+
%\sigma k^{\sigma-1}
%\frac{u^{2-\sigma}}{2-\sigma}|_0^1\\
&=
\frac{k^\sigma}{1-\sigma}+\frac{\sigma k^{\sigma-1}}{2-\sigma} \ ,
\end{align}
while
\begin{equation}
\int_{[0,1]}\left(\frac{k+1}{u}-1\right)^\sigma\,du
\leq
\int_{[0,1]}\left(\frac{k+1}{u}\right)^\sigma\,du
%=(k+1)^\sigma \frac{u^{1-\sigma}}{1-\sigma}|_0^1
=\frac{(k+1)^{\sigma}}{1-\sigma} \ .
\end{equation}
Therefore, by~\eqref{ineq:W1p-L1ineq} we arrive at the following inequalities
%since $max_{J_{b,k}} |\varphi_a'|=a(k/b)^{a-1}$
\begin{equation}
[g_{a,b}]_{W^{1,p},\varphi_a(J_{b,k})}^p
\leq
\left\{\begin{array}{ll}
a\left(\frac{b}{k}\right)^{1-a} k^{p(1-a)}\left(\frac{1}{1-p(1-a)}+\frac{p(1-a)}{k(2-p(1-a))}\right) & \ \ k \ \mbox{even} \\
a\left(\frac{b}{k}\right)^{1-a} (k+1)^{p(1-a)}\frac{1}{1-p(1-a)} & \ \ k \ \mbox{odd}
\end{array}\right. \ .
\end{equation}
Consequently, there exist positive real numbers 
$K_0(a,p)$ and $K(a,p)$, depending upon $a$ and $p$ only, such that 
\begin{align}
[g_{a,b}]_{W^{1,p},[0,1]}^p
&=\sum_{k=0}^{b-1}[g_{a,b}]_{W^{1,p},\varphi_a(J_{k,b})}^p\\
&\leq K_0(a,p)b^{1-a} \sum_{k=0}^{b-1} k^{(p-1)(1-a)}\\
%&\leq K(a,p)b^{(1-a)+(p-1)(1-a)+1}\\ 
&\leq K(a,p)b^{p(1-a)+1} \ .
\end{align}
This completes the proof.
\end{proof}
\begin{proof}[Proof of Theorem~\ref{mainthm:HolderSobolev} 3(a)--3(c)]
%%%%%%%%%%
(a)
Since the function $g_{a,b}$ is not $C^\alpha$ for any $\alpha>a$, and as the $\alpha$-H\"older condition is preserved under affine rescaling, it follows that
$f_a$ is also not $C^\alpha$ for any $\alpha>a$. 
Let us show that $f_a$ lies in $C^a$. 
By the H\"older rescaling principle~\cite[Proposition 2.2]{dFHazT2}
\begin{align}
[f_a]_{C^a,I_n}
&\leq 
[A_{I_n}^{-1}]_{\mathrm{Lip}} \; [g_{a,2n+1}]_{C^a, [0,1]} \; [A_{I_n}]_{\mathrm{Lip}}^a\\
&=
|I_n|^{1-a} \; [g_{a,2n+1}]_{C^a, [0,1]}\\
&\leq 2^{-n(1-a)} C(a,a) (2n+1)^{1+a} \ .
\end{align}
By the Proposition~\ref{prop:gluing}(ii), with $\omega(x)=x^a$, since $f_a$ fixed the endpoints of $I_n$ for each $n$, we find that
\begin{align}
[f_a]_{C^a, [0,1]}
\;\leq\; 
\sup_n \; [f_a]_{C^a, I_n}
\;\leq\; 
C(a,a)\sup_n 2^{-n(1-a)} (2n+1)^{1+a}
\;<\;
\infty
\end{align}
where, for the second inequality, we used Corollary~\ref{cor:q_a-Holder+Sobolev_bounds}(i).

\vspace{5pt}

\noindent
%%%%%%%%%%%
(b)
Since $f_a|_{I_n}=A_{I_n}^{-1}\circ g_{a,2n+1}\circ A_{I_n}$ 
an affine rescaling of a map differentiable Lebesgue-almost everywhere we find
that, for Lebesgue-almost every $x\in I_n$,
\begin{equation}
|f_a(x)|
%=|(A_{I_n}^{-1})'(g_{a,2n+1}(A_{I_n}(x)))| |g_{a,2n+1}'(A_{I_n}(x))| |A_{I_n}'(x)|
=|g_{a,2n+1}'(A_{I_n}(x))| \ .
\end{equation}
This, together with the change of variables formula for integrals and the observation that 
$|A_{I_n}'|= |I_n|^{-1}$, 
gives
\begin{align}
[f_a]_{W^{1,p},I_n}^p
&=\int_{I_n} |f_a'(x)|^p\,dx\\
&=|I_n|\int_{I_n}|g_{a,2n+1}'(A_{I_n}(x))|^p|A_{I_n}'(x)|\,dx\\
&=|I_n|\int_{A_{I_n}(I_n)}|g_{a,2n+1}'(u)|^p\,du\\
&\leq |I_n| \; [g_{a,2n+1}]_{W^{1,p}, [0,1]}^p
\end{align}
This, together with Corollary~\ref{cor:q_a-Holder+Sobolev_bounds}(ii), implies that
\begin{align}
[f_a]_{W^{1,p},[0,1]}^p
&=\sum_{n=1}^\infty [f_a]_{W^{1,p}, I_n}^p\\
&\leq \sum_{n=1}^{\infty} |I_n| \; [g_{a,2n+1}]_{W^{1,p}, [0,1]}^p\\
&\leq \sum_{n=1}^{\infty} 2^{-n} K(a,p) (2n+1)^{p(1-a)+1}
\end{align}
This last series is convergent.
Thus $f_a$ is $W^{1,p}$ for $1\leq p<(1-a)^{-1}$, as required.

\vspace{5pt}

\noindent
(c) 
Observe that $\varphi_a$ and $\varphi_a^{-1}$ are $C^a$, but they are not $C^\alpha$ for any $\alpha>a$. 
By the H\"older rescaling principle~\cite[Proposition 2.2]{dFHazT2}
\begin{equation}
[\psi_a]_{C^\alpha, I_n}
\;\leq\; 
[A_{I_n}^{-1}]_\mathrm{Lip} [\varphi_a]_{C^\alpha, [0,1]} [A_{I_n}]_\mathrm{Lip}^\alpha
\;=\;
|I_n|^{1-\alpha} [\varphi_a]_{C^\alpha, [0,1]}
\end{equation}
and a similar estimate holds for $[\psi_a^{-1}]_{C^\alpha,I_n}$.
Therefore, 
observing that for each $a\in (0,1]$, we have
$\psi_a|\partial I_n\equiv\mathrm{id}$, 
it follows that we may apply Proposition~\ref{prop:gluing}(ii). 
Hence, for all $\alpha\leq a$, $\psi_a$ is a bi-$\alpha$-H\"older homeomorphism.
%%%%%%%%%%%%%%%
\begin{comment}
\begin{equation}
[\psi_a]_{C^\alpha, [0,1]}
%\leq |[0,1]|^{1-\alpha}+2\sup_k [\psi_a]_{C^\alpha, I_k}
\leq 1+2\sup_k |I_k|^{1-\alpha} [\varphi_a]_{C^\alpha, [0,1]}
= 1+2^\alpha [\varphi_a]_{C^\alpha, [0,1]}
\end{equation}
\end{comment}
%%%%%%%%%%%%%
Since $f_a=\psi_a^{-1}\circ f_1\circ \psi_a$ and $f_1$ preserves Lebesgue measure $\mu$, 
%i.e. $f_*\mu(A)=\mu(f_1^{-1}A)=\mu(A)$ for all Borel $A$
it follows that the pullback $\psi_a^*\mu$ is an invariant measure for $f_a$.
As the functions $\varphi_a$ are absolutely continuous with respect to Lebesgue, it follows that $\psi_a$ is also absolutely continuous.
%%%%%%DENSITY%%%%%%%%
%However, the corresponding density $\varphi_a'$ is not $L^p$ for...
\end{proof}
%
%%%%%%%%%%%%%%%
\begin{comment}
\begin{remark}
In both examples there are infinitely many measures of maximal entropy.
In the first example Lebesgue measure is one of these measures and all other such measures are absolutely continuous with respect to Lebesgue.
In the second, all measures maximising entropy are singular.
\end{remark}
\end{comment}
%%%%%%%%%%%%%
For $a\neq 1$, 
the above analysis can also be applied to the construction when restricting to the union of the intervals 
$I_n, I_{n+1},\ldots$, just as in the case of Corollary~\ref{cor:a=1_accum_on_id}.
Thus, analogously to that corollary, we also get the following result.
\begin{corollary}\label{cor:aneq1_accum_on_id}
For each $a\in (0,1)$, 
there exists a sequence $f_{a,n}\in C^0([0,1],[0,1])$ satisfying properties 3(a)-3(c) in Theorem~\ref{mainthm:HolderSobolev} above and
with the additional property that $\lim_{n\to\infty} f_n= \mathrm{id}$ where convergence is taken
\begin{itemize}
\item
in the $C^\alpha$-topology for any $\alpha\in(0,a)$,
\item
in the $W^{1,p}$-topology for any $p\in\left[1,(1-a)^{-1}\right)$.
\end{itemize}

\end{corollary}
%
%
%%%%%%%%%%%%%%%%%%%%%%%%%%%%%%%%%%%%%%%%%%%%%%%
\section{Examples in the little Zygmund class.}\label{sect:Zygmund}
As was previously remarked, the function $f_1$ constructed in 
Theorem~\ref{mainthm:HolderSobolev} above
has modulus of continuity $t\log (\frac{1}{t})$, 
and hence is $\alpha$-H\"older for every $\alpha\in [0,1)$, but it is not Lipschitz.
Further, it does not satisfy either the big or little Zygmund conditions.
Recall that a continuous function $f$ of the interval $[0,1]$ satisfies the {\it big Zygmund condition} if, for all $x$ in $(0,1)$,
\begin{equation}
\left|f(x+t)+f(x-t)-2f(x)\right|=O(t)
\end{equation}
and the {\it little Zygmund condition} if, for all $x$ in $(0,1)$,
\begin{equation}
\left|f(x+t)+f(x-t)-2f(x)\right|=o(t) \ .
\end{equation}
Observe that this condition only makes sense at interior points.
We will denote the sets of functions satisfying the big and little Zygmund conditions respectively by
$C^Z([0,1],\mathbb{R})$ and $C^{z}([0,1],\mathbb{R})$.
Observe that these are both linear spaces.
Define the {\it Zygmund semi-norm} by
\begin{equation}
[f]_{Z,[0,1]}=\sup_{t: x\pm t\in [0,1]}\sup_{x\in[0,1]}\frac{|f(x+t)+f(x-t)-2f(x)|}{|t|} \ .
\end{equation}
Then
\begin{equation}
\|f\|_{C^Z([0,1],\mathbb{R})}=\|f\|_{C^0([0,1],\mathbb{R})}+[f]_{Z,[0,1]}
\end{equation}
\vspace{0pt}

\noindent
defines a complete norm on $C^Z([0,1],\mathbb{R})$, which we call the {\it Zygmund norm}.
With this topology, $C^z([0,1],\mathbb{R})$ is a closed subspace of $C^Z([0,1],\mathbb{R})$.

The Zygmund classes strictly (in fact compactly) contain the Lipschitz class and are contained in the $\alpha$-H\"older class for each $\alpha\in[0,1)$.
Moreover, by a theorem of Zygmund~\cite{ZygmundBook}, functions in the big Zygmund class have modulus of continuity $t\log (\frac{1}{t})$.
(See~\cite{ZygmundBook} for more on these classes.) 
The reason that the example $f$ given above is not in either Zygmund class is that, 
at a turning point $x$ of a $k$-branched horseshoe, for all $t$ sufficiently small, 
\begin{equation}
\left|(f(x+t)-f(x))-(f(x)-f(x-t))\right|\geq 2kt \ .
\end{equation} 
One might think that, by replacing piecewise-affine 
with smooth horseshoes and possibly changing the 
lengths of the intervals, one may be able to improve the regularity, say to the big Zygmund class.
However, this is not possible as, for {\it any} 
$k$-branched horseshoe, there is {\it some} turning point $x$ and some $t$ so that
\begin{equation}
\left|(f(x+t)-f(x))-(f(x)-f(x-t))\right|\geq kt \ .
\end{equation}
%%%%%%%%%%%%%%%
\begin{comment}
However, the following Theorem shows that there are examples in the little Zygmund class with infinite topological entropy.
%
\begin{theorem}\label{thm:interval_zyg}
There exists $f\in C^0([0,1],[0,1])$ with the following properties
\begin{enumerate}
\item
$f$ satisfies the little Zygmund condition
\item
$h_\mathrm{top}(f)=+\infty$
\end{enumerate}
\end{theorem}
\end{comment}
%%%%%%%%%%%%%
%
\begin{remark}
It will be clear from the construction below that piecewise-affine 
(with countably many pieces, as in the preceding example) can also be constructed but, 
necessarily, these can only lie in the big Zygmund class. 
\end{remark}
\begin{figure}[htp]
\begin{center}
\psfrag{0}[][]{ ${\scriptstyle 0}$}
\psfrag{1}[][]{ ${\scriptstyle 1}$}
\psfrag{x}[][][1]{$x$}
\psfrag{f}[][][1]{$\ \ f(x)$}
\psfrag{J0}[][][1]{$J_0$}
\psfrag{J1}[][][1]{$J_1$}
\psfrag{J2}[][][1]{$J_2$}
\psfrag{I1}[][][1]{$I_1$}
\psfrag{I2}[][][1]{$I_2$}
\psfrag{I3}[][][1]{$I_3$}
\psfrag{J0'}[][][1]{$J_0'$}
\psfrag{J1'}[][][1]{$J_1'$}
\psfrag{J2'}[][][1]{$J_2'$}
\psfrag{I1'}[][][1]{$I_1'$}
\psfrag{I2'}[][][1]{$I_2'$}
\psfrag{I3'}[][][1]{$I_3'$}

\includegraphics[width=4.5in]{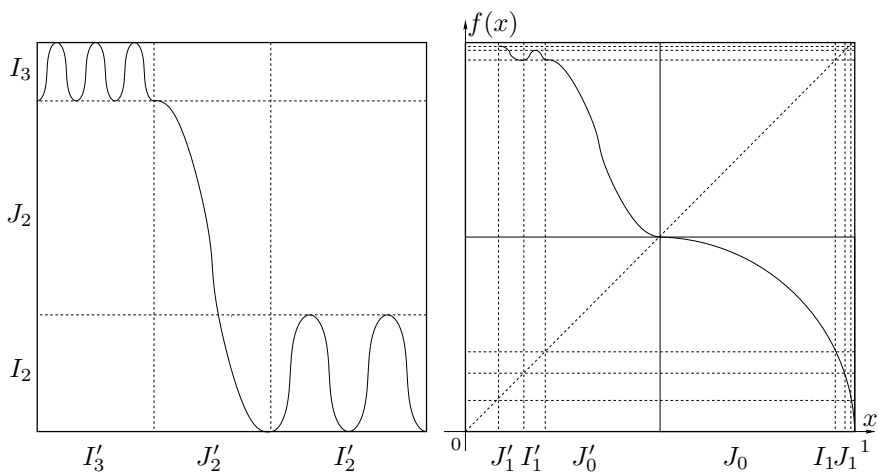}
\end{center}
\caption[slope]{\label{figendo3} 
The graph of the little Zygmund interval 
endomorphism $f$ on $[0,1]$ with infinite 
topological entropy (right) 
and an affine rescaling of the graph on 
the interval $I_3'\cup J_3'\cup I_2'$ (left).}
\end{figure}
Before starting the construction proper, 
we introduce the following notation.
Given oriented closed intervals $I$ and $J$ denote by $A_{I,J}$ the unique orientation-preserving affine bijection from $J$ to $I$,
and $A_{I,J}^{-}$ the unique orientation-reversing affine bijection from $J$ to $I$.
Given a positive integer $n$, let 
$g_n\colon [0,1]\to [0,\frac{1}{n}]$ 
be defined by
\begin{equation}
g_n(x)=\frac{1}{2n}\left(1-\cos(2\pi n x)\right) \ .
\end{equation}
Observe that, for each $n$,
$g_n$ maps the interval $[0,1]$ onto $[0,\frac{1}{n}]$ 
in a $2n$-to-$1$ manner.
Also, $g_n$ is Lipschitz with Lipschitz 
constant $\pi$, and is differentiable 
with vanishing derivative at the endpoints. 
It also has the following important property.

\vspace{5pt}

\begin{quote}
{\sl Key Property:}
Let $\phi\colon [0,\frac{1}{n}]\to [0,1]$ be any homeomorphism.
Then the composition $\phi\circ g_n\colon [0,1]\to [0,1]$ is a $2n$-branched horseshoe. 
\end{quote}

\vspace{5pt}

\noindent
Now we get to the construction proper. 
Take any continuous map $f\colon [\frac{1}{2},1]\to [0,\frac{1}{2}]$ satisfying the following properties
\begin{enumerate}
\item
$f$ is an orientation-reversing homeomorphism which is little Zygmund, with little Zygmund inverse, 
\item
$f'(\frac{1}{2})$ exists and is zero,
\item 
$f'(1)$ exists (in the extended sense) and equals $-\infty$,
\item
There exist pairwise disjoint open subintervals $I_1, I_2,\ldots\subset[\frac{1}{2},1]$, ordered from left to right, 
which converge to $\{1\}$ and have the property that, for all $n\in\mathbb{N}$,
\begin{equation}
\frac{|f(I_{n})|}{|I_{n}|}=n \ .
\end{equation}
\item
Let $J_0,J_1,J_2,\ldots\subset [0,1]$ denote 
the pairwise disjoint closed intervals which 
are connected components of 
$[\frac{1}{2},1]\setminus \bigcup_{n=1}^\infty I_n$, 
again ordered from left to right.
Then, for all $n\in\mathbb{N}$,
\begin{equation}\label{property:J-difference-quotients}
\frac{|I_n|+|J_n|}{|f(J_n)|}\leq K \ .
\end{equation}
\end{enumerate}
\begin{remark}
Observe that such functions are easy to construct.
For example, if we start with the function 
$\phi\colon [0,\tfrac{1}{e}]\to [0,\tfrac{1}{e}]$ 
given by $\phi(x)=x\log\left(\frac{1}{x}\right)$,
then $\phi$ has 
infinite derivative at zero, 
a critical point at $\frac{1}{e}$, and 
satisfies the Zygmund condition.
Moreover, so do any affine rescalings.
Thus, setting 
\begin{equation}
f
=
A_{\left[0,\tfrac{1}{e}\right],\left[0,\tfrac{1}{2}\right]}\circ \phi \circ A_{\left[\tfrac{1}{2},1\right],\left[0,\tfrac{1}{e}\right]}^{-} \ ,
\end{equation} 
%where $A_{I,J}^{-}$ denotes the orientation-reversing affine map from $I$ to $J$, 
we find that $f$ satisfies properties (1)--(3).  
We can choose intervals $I_n$ to be a small neighbourhood about the unique point $x_n$ where $f'(x_n)=-n$, for each $n$, so that (4) is satisfied.
By taking $|I_n|$ sufficiently small for each $n$, this will ensure that 
$\frac{|I_n|+|J_n|}{|f(J_n)|}\leq 2\frac{|J_n|}{|f(J_n)|}$.
As $|f'(y)|$ tends to infinity as $y$ tends to $1$, the property (5) holds.  
\end{remark}
Next, extend $f$ to a continuous map from $[0,1]$ to $[0,1]$ as follows.
First, however, we need the following notation.
Recall that $\nu(x)=1-x$ is the unique affine orientation-reversing map of $[0,1]$.
Define $\sigma\colon [0,1]\to [0,1]$ by
\begin{equation}
\sigma(x)=\tfrac{1}{2}(1-\cos(\pi x)) \ .
\end{equation}
(More generally, $\sigma$ can be any smooth orientation-preserving homeomorphism of $[0,1]$ with bounded derivative and with critical points at $0$ and $1$.) 

For each $n$, define $I_n'=f(I_n)$ and $J_n'=f(I_n)$.
Observe that, since $f|_{[\frac{1}{2},1]}$ is an orientation-reversing homeomorphism, 
the $I_n'$ are decreasing sequence of pairwise disjoint subintervals, converging to $\{0\}$, and are interlaced by the $J_n'$.
We define $f$ on each $I_n'$ and $J_n'$ separately, ensuring that the resulting map is continuous.
First, define $f$ on $I_n'$, for each integer $n\geq 1$, by
\begin{equation}
f(x)=A_{[0,\frac{1}{n}],I_n} \circ g_{n}\circ A_{I_n',[0,1]}(x) \qquad x\in I_n', \ n\geq 1 \ .
\end{equation}
Observe that $f$ is differentiable on each $I_n'$ and has vanishing derivative at the endpoints.
Define $f$ on $\bigcup_{n=0}^{\infty} J_n'$ so that $f$ is differentiable on each $J_n'$ with vanishing derivatives at the endpoints. 
Namely, we set
\begin{equation}
f(x)
=
\left\{\begin{array}{ll}
A_{[0,1],I_n\cup J_n}\circ\nu\circ\sigma\circ A_{J_n',[0,1]}(x)  & \ x\in J_n', \ n\geq 1\\
A_{[0,1],J_0}\circ \nu\circ\sigma\circ A_{J_0',[0,1]}(x) & \ x\in J_0'
\end{array}\right. \ .
\end{equation}
\begin{proof}[Proof of Theorem~\ref{mainthm:Zygmund}]
Observe that $f|_{I_n'}$ are differentiable with uniformly bounded derivative over all $n$. 
In fact, for all $n\in\mathbb{N}$,
\begin{equation}
\|f'\|_{C^0(I_n',\mathbb{R})}
\leq 
\pi \ .
\end{equation}
Also the $f|_{J_n'}$ are differentiable and, by inequality~\eqref{property:J-difference-quotients} above, 
the derivatives are also uniformly bounded over all $n\neq 0$ since, for all integers $n\geq 1$,
\begin{equation}
\|f'\|_{C^0(J_n',\mathbb{R})}
\;=\;
\frac{|I_n\cup J_n|}{|J_n'|}\|\sigma'\|_{C^0([0,1],\mathbb{R})}
\;\leq\; 
K\pi/2 \ .
\end{equation}
Since $f|J_0$ is also differentiable, and on all intervals $I_n'$ and $J_n'$, $f'$ vanishes at the endpoints, 
it follows that $f|_{(0,\frac{1}{2}]}$ is also differentiable.
Hence $f|_{[0,\frac{1}{2}]}$ is little Zygmund.
Next, as $f|_{[\frac{1}{2},1]}$ is little Zygmund and $f$ is differentiable from the left- and from the right at $x=\frac{1}{2}$ (with zero derivative),
it follows that $f$ is little Zygmund on $[0,1]$.

Let us now show that $f$ has infinite topological entropy.
Since $f(I_n)=I_n'$ and $f(I_n')=I_n$ it follows that $I_n'$ is $f^2$-invariant, for each $n$.
Moreover $f|_{I_n}$ is a homeomorphism from $I_n$ to $I_n'$, and $f|_{I_n'}$ is an affine rescaling of $g_{n}$. 
Thus by the Key Property stated above, $f^2|_{I_n}$ is a $2n$-branched horseshoe. 
Hence $f^2|_{I_n}$ has topological entropy at least $\log (2n)$.
Consequently
\begin{equation}
h_\mathrm{top}(f^2)
\;\geq\; 
\sup_{n}h_\mathrm{top}\left(f^2|_{I_n'}\right)
\;=\;
\sup_n \log (2n)
\;=\;
\infty \ .
\end{equation}
Finally, since it is known that for an arbitrary continuous self-map $F$ of a 
compact metric space the equality $h_\mathrm{top}(F^k)=k h_\mathrm{top}(F)$ holds
for any positive integer $k$, it follows that $h_\mathrm{top}(f)=+\infty$ as well.
\end{proof}

%%%%%%%%%%%%%%%%%%%%%%%%%%%%%%%%%%%%%%%%%%
\section{Concluding Remarks.}\label{sect:conclusion}
%%%%%%%%%%%%%%%%%%%%%%%%%%%%%%%%%%%%%%%%%%
We finish with a number of open problems suggested by this work.
\begin{enumerate}
\item
Do there exist one-dimensional endomorphisms with infinite topological entropy which are asymptotically $h$-expansive, or even $h$-expansive?
What is their `optimal' regularity: can they be H\"older or Sobolev?
\item
The Zygmund example has the property that its second iterate is no longer Zygmund. 
(It is of the type given in Theorem~\ref{mainthm:HolderSobolev} which, 
as remarked on in Section~\ref{sect:Zygmund} cannot lie in the Zygmund class.)
Does there exist a map in the Zygmund class (big or little) with infinite topological entropy, 
and such that all iterates are also in the Zygmund (big or little) class?
\item
(Alby Fisher) 
Is there an ergodic example of a map with infinite entropy in dimension-one? 
With some H\"older, Sobolev or Zygmund class?
\end{enumerate}

\subsection*{Acknowledgements}
The author would like to thank IME-USP, ICERM (Brown University), 
%Imperial College London, and the CUNY Graduate Center 
Uppsala University and KTH Stockholm for their hospitality.
I would also like to thank E.\ de Faria, C.\ Tresser and M.\ Benedicks for many stimulating conversations 
regarding the relation between regularity and entropy.
Finally, I also thank 
%C.\ Tresser and 
M.\ Benedicks for the question that lead to this work.

\end{document}